\documentclass[12pt]{amsart}

\usepackage{amsmath,amssymb,amsthm}
\usepackage[all]{xy}
\usepackage{color}
\usepackage[dvipdfm,colorlinks=true]{hyperref}

\numberwithin{equation}{section}

\SelectTips{eu}{12}

\newtheorem{theorem}{Theorem}[section]

\newtheorem{lemma}[theorem]{Lemma}
\newtheorem{corollary}[theorem]{Corollary}

\theoremstyle{definition}

\newtheorem{example}[theorem]{Example}

\theoremstyle{remark}

\newcommand{\Z}{\mathbb{Z}}
\newcommand{\SO}{\mathrm{SO}}
\newcommand{\SU}{\mathrm{SU}}
\newcommand{\Sp}{\mathrm{Sp}}

\title{Samelson products in $p$-regular $\SO(2n)$ and its homotopy normality}

\author{Daisuke Kishimoto}
\address{Department of Mathematics, Kyoto University, Kyoto, 606-8502, Japan}
\email{kishi@math.kyoto-u.ac.jp}
\author{Mitsunobu Tsutaya}
\address{Department of Mathematics, Kyoto University, Kyoto, 606-8502, Japan}
\email{tsutaya@math.kyoto-u.ac.jp}
\thanks{D.K. is supported by JSPS KAKENHI (No. 25400087)}

\begin{document}

\maketitle

\baselineskip 16pt

\begin{abstract}
A Lie group is called $p$-regular if it has the $p$-local homotopy type of a product of spheres. (Non)triviality of the Samelson products of the inclusions of the factor spheres into $p$-regular $\SO(2n)_{(p)}$ is determined, which completes the list of (non)triviality of such Samelson products in $p$-regular simple Lie groups. As an application, we determine the homotopy normality of the inclusion $\SO(2n-1)\to\SO(2n)$ in the sense of James at any prime $p$.
\end{abstract}


\section{Introduction and statement of the results}
\label{section_intro}
Let $G$ be a compact connected Lie group. By the classical result of Hopf, it is well known that there is a rational homotopy equivalence
$$G\simeq_{(0)}S^{2n_1-1}\times\cdots\times S^{2n_\ell-1}$$
where $n_1\le\cdots\le n_\ell$. The sequence $n_1\le\cdots\le n_\ell$ is called the type of $G$. Here is the list of the types of simple Lie groups.

\renewcommand{\arraystretch}{1.2}
\begin{table}[htbp]
\centering
\begin{tabular}{l|l||l|l}
\hline
$\mathrm{SU}(n)$&$2,3,\ldots,n$&$\mathrm{G}_2$&$2,6$\\
$\SO(2n+1)$&$2,4,\ldots,2n$&$\mathrm{F}_4$&$2,6,8,12$\\
$\mathrm{Sp}(n)$&$2,4,\ldots,2n$&$\mathrm{E}_6$&$2,5,6,8,9,12$\\
$\SO(2n)$&$2,4,\ldots,2n-2,n$&$\mathrm{E}_7$&$2,6,8,10,12,14,18$\\
&&$\mathrm{E}_8$&$2,8,12,14,18,20,24,30$\\\hline
\end{tabular}
\end{table}

\noindent Serre generalizes the above rational homotopy equivalence to a $p$-local homotopy equivalence such that when $G$ is semisimple and $G_{(p)}$ is simply connected, there is a $p$-local homotopy equivalence 
\begin{equation}
\label{p-regular}
G\simeq_{(p)}S^{2n_1-1}\times\cdots\times S^{2n_\ell-1}
\end{equation}
if and only if $p\ge n_\ell$, in which case $G$ is called $p$-regular. In this paper we are interested in the standard multiplicative structure of the $p$-localization $G_{(p)}$ when $G$ is $p$-regular, and then we assume that $G$ is a simple Lie group in the above table and is $p$-regular throughout this section. Recall that for a homotopy associative H-space $X$ with inverse and maps $\alpha\colon A\to X,\beta\colon B\to X$, the correspondence
$$A\wedge B\to X,\quad(x,y)\mapsto\alpha(x)\beta(y)\alpha(x)^{-1}\beta(y)^{-1}$$
is called the Samelson product of $\alpha,\beta$ in $X$ and is denoted by $\langle\alpha,\beta\rangle$. One easily sees that in investigating the multiplicative structure of $G_{(p)}$, the Samelson products $\langle\epsilon_i,\epsilon_j\rangle$ play the fundamental role as in \cite{KK}, where $\epsilon_i$ is the inclusion $S^{2n_i-1}\to S^{2n_1-1}_{(p)}\times\cdots\times S^{2n_\ell-1}_{(p)}\simeq G_{(p)}$ into the $i$-th factor. So it is our task to determine (non)triviality of these Samelson products.
In this direction, Bott \cite{B} studied the order of a certain class of Samelson products in $\SU(n)$ and $\Sp(n)$, for example.

We here make a remark on the choice of $\epsilon_i$ which depends on the $p$-local homotopy equivalence \eqref{p-regular}. Recall from \cite[Theorem 13.4]{T} that 
\begin{equation}
\label{pi(S)}
\pi_*(S^{2m-1}_{(p)})=0\quad\text{for}\quad 2m-1<*<2m+2p-4.
\end{equation}
Then we see that $\pi_{2n_i-1}(G_{(p)})$ is a free $\Z_{(p)}$-module for all $i$, and so $\pi_{2n_i-1}(G_{(p)})\cong\Z_{(p)}$ for all $i$ and $G\ne\SO(2n)$ since the entries of the type are distinct for $G\ne\SO(2n)$ as in the above table. Hence for $G\ne\SO(2n)$ we may choose any generator of $\pi_{2n_i-1}(G_{(p)})\cong\Z_{(p)}$ as $\epsilon_i$. For $G=\SO(2n)$ we will make an explicit choice of $\epsilon_i$ below.

We first consider the Samelson products $\langle\epsilon_i,\epsilon_j\rangle$ in $G_{(p)}$ when $G$ is the classical group execpt for $\SO(2n)$.

\begin{theorem}
\label{classical}
Let $G$ be the $p$-regular classical group except for $\SO(2n)$, and let $\epsilon_i$ be a generator of $\pi_{2n_i-1}(G_{(p)})\cong\Z_{(p)}$ for the type $\{n_1,\ldots,n_\ell\}$ of $G$. Then 
$$\langle\epsilon_i,\epsilon_j\rangle\ne 0\quad\text{if and only if}\quad n_i+n_j>p.$$
\end{theorem} 

\begin{proof}
If $G=\mathrm{SU}(n),\mathrm{Sp}(n)$, nontriviality of the Samelson products follows from the result of Bott \cite{B} and triviality follows from the fact that $\pi_{2*}(G_{(p)})=0$ for $*<p$ which is deduced from \eqref{pi(S)}. Since there is a homotopy equivalence as loop spaces $\mathrm{Sp}(n)_{(p)}\simeq\SO(2n+1)_{(p)}$ due to Friedlander \cite{F}, the case of $\SO(2n+1)_{(p)}$ is the same as $\mathrm{Sp}(n)_{(p)}$.
\end{proof}

We next consider the Samelson products $\langle\epsilon_i,\epsilon_j\rangle$ in $G_{(p)}$ when $G$ is the exceptional Lie group. Some of these Samelson products are calculated in \cite{HK2,KK}, and (non)triviality of all these Samelson products is determined in \cite{HKO} as follows. 

\begin{theorem}
[\cite{HKO}]
\label{HKO}
Let $G$ be a $p$-regular compact connected exceptional simple Lie group, and let $\epsilon_i$ be a generator of $\pi_{2n_i-1}(G_{(p)})\cong\Z_{(p)}$ for the type $\{n_1,\ldots,n_\ell\}$ of $G$. Then
$$\langle\epsilon_i,\epsilon_j\rangle\ne 0\quad\text{if and only if}\quad n_i+n_j=n_k+p-1\text{ for some }k.$$
\end{theorem}

Thus the only remaining case is $\SO(2n)$.
The purpose of this paper is to show that a sufficient condition for nontriviality of the Samelson products $\langle\epsilon_i,\epsilon_j\rangle$ in $G_{(p)}$ (Lemma \ref{criterion}) used in \cite{KO,HK1,HK2,HKO} is actually a necessary and sufficient condition, and is to apply it to determination of (non)triviality of all the Samelson products $\langle\epsilon_i,\epsilon_j\rangle$ in $\SO(2n)_{(p)}$.
The difficulty of this case is caused by the middle dimensional sphere $S^{2n-1}_{(p)}$ in $\SO(2n)_{(p)}$ which vanishes by the inclusion $\SO(2n)\to\SO(2n+1)$.
we choose the maps $\epsilon_i$. Let $\epsilon_i$ be the composite 
$$S^{4i-1}\to\SO(2n-1)_{(p)}\xrightarrow{\rm incl}\SO(2n)_{(p)}$$ 
for $i=1,\ldots,n-1$, where the first arrow is a generator of $\pi_{4i-1}(\SO(2n-1)_{(p)})\cong\Z_{(p)}$.
Let $\theta\colon S^{2n-1}\to\SO(2n)_{(p)}$ be the map corresponding to the adjoint of the fiber inclusion of the canonical homotopy fiber sequence
$$S^{2n}\to B\SO(2n)\to B\SO(2n+1).$$
There are only two results on Samelson products in $\SO(2n)$ involving $\theta$: Mahowald \cite{Ma} showed that the Samelson product $\langle\theta,\theta\rangle\in\pi_{4n-2}(\SO(2n))$ has order $(2n-1)!/8$ or $(2n-1)!/4$ according as $n$ is even or odd.
Hamanaka and Kono \cite{HK1} showed that the Samelson product $\langle\epsilon_{\tfrac{p-1}{2}},\theta\rangle\in\pi_{4n-2}(\SO(2n)_{(p)})$ is non-trivial when $p\le 2n-1$.
Our main result determines (non)triviality of all Samelson products of $\epsilon_i$ and $\theta$ in $p$-regular $\SO(2n)$.

\begin{theorem}
\label{main}
Let $\epsilon_i,\theta$ be the above maps into $\SO(2n)_{(p)}$ for $p$-regular $\SO(2n)$. All nontrivial Samelson products of $\epsilon_i,\theta$ in $\SO(2n)_{(p)}$ are
$$\langle\epsilon_i,\epsilon_j\rangle\quad\text{for}\quad 2i+2j>p\quad\text{and}\quad\langle\epsilon_{n-1},\theta\rangle=\langle\theta,\epsilon_{n-1}\rangle,\;\langle\theta,\theta\rangle\quad\text{for}\quad p=2n-1.$$
\end{theorem}

Recall that an H-map $f\colon X\to Y$ between homotopy associative H-spaces with inverse $X,Y$ is homotopy normal in the sense of James \cite{J} if the Samelson product $\langle f,1_Y\rangle$ compressed to $X$ through $f$ up to homotopy. This is a generalization of the inclusion of a normal subgroup.
James proved that $\operatorname{O}(n)$ is not homotopy normal in $\operatorname{O}(n+1)$ when $n\ge 2$ using the mod $2$ cohomology.
His proof implies that the $2$-localization $\SO(n)_{(2)}$ is not homotopy normal in $\SO(n+1)_{(2)}$ when $n\ge 2$.
As an application of Theorem \ref{main} we will prove:

\begin{theorem}
\label{normal}
The inclusion $\iota_{(p)}\colon\SO(2n-1)_{(p)}\to\SO(2n)_{(p)}$ is homotopy normal if and only if $p>2n-1$.
\end{theorem}

For $p>2n-1$, we can prove the following stronger result.

\begin{theorem}
\label{H-type}
For $p>2n-1$, the map $\iota_{(p)}\cdot\theta\colon\SO(2n-1)_{(p)}\times S^{2n-1}_{(p)}\to\SO(2n)_{(p)}$ is an H-equivalence, where $S^{2n-1}_{(p)}$ is a homotopy associative and homotopy commutative H-space.
\end{theorem}


\section{Detecting Samelson products by the Steenrod operations}

Let $G$ be a $p$-torsion free connected finite loop space of type $n_1\le\cdots\le n_\ell$ throughout this section where the type of a finite loop space is similarly defined. We set notation for $G$. Since $G$ is $p$-torsion free, we have
$$H^*(BG_{(p)};\Z/p)=\Z/p[x_1,\ldots,x_\ell],\qquad|x_i|=2n_i.$$
We fix this presentation of the mod $p$ cohomology of $BG_{(p)}$. Note that 
$$H^*(G_{(p)};\Z/p)=\Lambda(e_1,\ldots,e_\ell)$$ 
for the suspension $e_i$ of $x_i$. For each $i$, we take $\epsilon_i\in\pi_{2n_i-1}(G_{(p)})$ which is not divisible by non-units in $\mathbb{Z}_{(p)}$ such that 
\begin{equation}
\label{epsilon}
(\Sigma\epsilon_i)^*\circ\iota_1^*(x_j)=
\begin{cases}
h_i\Sigma u_{2n_i-1}&i=j\\
0&i\ne j
\end{cases}
\end{equation}
for some $h_i\in\mathbb{Z}_{(p)}$, where $\iota_1\colon\Sigma G_{(p)}\to BG_{(p)}$ is the canonical map and $u_k$ is a generator of $H^k(S^k;\mathbb{Z}_{(p)})\cong\mathbb{Z}_{(p)}$.
The following lemma is first used in \cite{KO} and is the main tool in the proof of Theorem \ref{HKO} given in \cite{HKO}. Here we reproduce the proof for completeness of the present paper.

\begin{lemma}
[{\cite[Proof of Theorem 1.1]{KO}}]
\label{criterion}
Suppose that $h_i$ and $h_j$ are units in $\mathbb{Z}_{(p)}$.
If $\mathcal{P}^1x_k$ is decomposable and includes the term $cx_ix_j$ ($c\ne 0$), the Samelson product $\langle\epsilon_i,\epsilon_j\rangle$ is nontrivial.
\end{lemma}

\begin{proof}
Suppose $\langle\epsilon_i,\epsilon_j\rangle=0$ under the assumption that $\mathcal{P}^1x_k$ includes the term $cx_ix_j$ ($c\ne 0$). Let $\bar{\epsilon}_m\colon S^{2n_m}\to BG_{(p)}$ be the adjoint of $\epsilon_m$. Then by \eqref{epsilon}, we have $\bar{\epsilon}_m^*(x_m)=h_mu_{2m}$. By adjointness of Samelson products and Whitehead products, the Whitehead product $[\bar{\epsilon}_i,\bar{\epsilon}_j]$ in $BG_{(p)}$ is trivial, and then there is a map $\mu\colon S^{2n_i}\times S^{2n_j}\to BG_{(p)}$ satisfying $\mu\vert_{S^{2n_i}\vee S^{2n_j}}=\bar{\epsilon}_i\vee\bar{\epsilon}_j$. So we get $\mu^*(x_i)=h_i(u_{2n_i}\otimes 1)$ and $\mu^*(x_j)=h_j(1\otimes u_{2n_i})$, and hence 
$$ch_ih_ju_{2n_i}\otimes u_{2n_j}=\mu^*(cx_ix_j)=\mu^*(\mathcal{P}^1x_k)=\mathcal{P}^1\mu^*(x_k)=0$$
where the second and the last equality follows from the decomposability of $\mathcal{P}^1x_k$ and triviality of $\mathcal{P}^1$ on $H^*(S^{2n_i}\times S^{2n_j};\Z/p)$, respectively. This is a contradiction to $ch_ih_j\ne0$.
\end{proof}

In this lemma, the assumption on the decomposability of $\mathcal{P}^1x_k$ cannot be removed.
Here is a counterexample.

\begin{example}
We consider $\SU(4)$ at the prime $3$.
Recall that $H^\ast(B\SU(4);\mathbb{Z}/3)=\mathbb{Z}/3[c_2,c_3,c_4]$, where $c_i$ denotes the $i^\text{th}$ universal Chern class.
By inspection, we have
$$\mathcal{P}^1c_2=c_2^2+c_4.$$
For a degree reason, the inclusion $\epsilon_1\colon S^3=\SU(2)\to\SU(4)$ satisfies $(\Sigma\epsilon_1)^*\circ\iota_1^*(c_2)=\Sigma u_3$ as in \eqref{epsilon}, but the Samelson product $\langle\epsilon_1,\epsilon_1\rangle$ is trivial since $\SU(2)$ commutes up to homotopy with itself in $\SU(4)$.
\end{example}

We elaborate Lemma \ref{criterion} to prove that its converse is true when $G_{(p)}$ is a product of spheres.
The following lemma is useful to detect the nontriviality of a Samelson product when $G_{(p)}$ is decomposed into a product of a sphere and some space.
The proof is independent of Lemma \ref{criterion}.

\begin{lemma}
\label{criterion2}
For integers $1\le i,j,k\le\ell$, suppose that there is a map $\pi_k\colon G_{(p)}\to S^{2n_k-1}_{(p)}$ such that $\pi_k^\ast (u_{2n_k-1})=e_k$,  $h_i$ and $h_j$ are units in $\mathbb{Z}_{(p)}$, and $n_i+n_j=n_k+p-1$.
Then $\pi_k\circ\langle\epsilon_i,\epsilon_j\rangle\ne0$ if and only if $\mathcal{P}^1x_k$ includes the term $cx_ix_j$ with $c\ne 0$.
\end{lemma}

\begin{proof}
We prove both implications simultaneously.
We may suppose that $h_i=h_j=h_k=1$.
Let $P^2G_{(p)}$ be the projective plane of $G_{(p)}$, i.e. there is a cofiber sequence
\begin{equation}
\label{P^2G}
\Sigma G_{(p)}\wedge G_{(p)}\xrightarrow{H}\Sigma G_{(p)}\xrightarrow{\rho_1}P^2G_{(p)}
\end{equation}
where $H$ is the Hopf construction. Put $\bar{x}_i=\iota_2^*(x_i)$ for the natural map $\iota_2\colon P^2G_{(p)}\to BG_{(p)}$. Since $\iota_1=\iota_2\circ\rho_1$ for the inclusion $\rho_1\colon\Sigma G_{(p)}\to P^2G_{(p)}$ and $\iota_1^*(x_i)=\Sigma e_i$, we have $\rho_1^*(\bar{x}_i)=\Sigma e_i$. By \cite[Section 3]{L}, we also have $\delta_1^*(\Sigma^2e_i\otimes e_j)=\bar{x}_i\bar{x}_j$ for the connecting map $\delta_1\colon P^2G_{(p)}\to\Sigma^2G_{(p)}\wedge G_{(p)}$ of the cofiber sequence \eqref{P^2G}. Consider the map
$$\Phi=\Sigma\langle\epsilon_i,\epsilon_j\rangle-[\Sigma\epsilon_i,\Sigma\epsilon_j]\colon\Sigma S^{2n_i-1}\wedge S^{2n_j-1}\to\Sigma G_{(p)}$$
where $[-,-]$ denotes the Whitehead product. The map $\Phi$ is connected with the Hopf construction $H$ through the map constructed by Morisugi \cite[Theorem 5.1]{Mo} such that there is a map $\xi\colon S^{2n_i-1}\wedge S^{2n_j-1}\to G_{(p)}\wedge G_{(p)}$ satisfying
$$\Phi=H\circ\Sigma\xi\quad\text{and}\quad\xi^*(e_i\otimes e_j)=u_{2n_i-1}\otimes u_{2n_j-1}.$$
Then we get a homotopy commutative diagram
$$\xymatrix{\Sigma G_{(p)}\ar[r]^{\rho_2}\ar@{=}[d]&C_\Phi\ar[r]^(.3){\delta_2}\ar[d]^{\lambda_1}&\Sigma^2 S^{2n_i-1}\wedge S^{2n_j-1}\ar[d]^{\Sigma^2\xi}\\
\Sigma G_{(p)}\ar[r]^{\rho_1}&P^2G_{(p)}\ar[r]^(.45){\delta_1}&\Sigma^2 G_{(p)}\wedge G_{(p)}}$$
whose rows are homotopy cofibrations, implying that
\begin{equation}
\label{lambda1}
\rho_2^*\circ\lambda_1^*(\bar{x}_k)=\Sigma e_k\quad\text{and}\quad\lambda_1^*(\bar{x}_i\bar{x}_j)=\delta_2^*(\Sigma^2u_{2n_i-1}\otimes u_{2n_j-1}).
\end{equation}
We have
$$\pi_k\circ\langle\epsilon_i,\epsilon_j\rangle=c\alpha_1\quad(c\in\Z/p)$$ 
where $\alpha_1$ is a generator of $\pi_{2n_k+2p-4}(S^{2n_k-1})\cong\Z/p$ \cite[Proposition 13.6]{T}.
Note that $\pi_k\circ\langle\epsilon_i,\epsilon_j\rangle$ is nontrivial if and only if $c\ne0$.
Then for the map 
$$\widehat{\Phi}=c\Sigma\alpha_1-[\Sigma\pi_k\circ\epsilon_i,\Sigma\pi_k\circ\epsilon_j]\colon\Sigma S^{2n_i-1}\wedge S^{2n_j-1}\to\Sigma S^{2n_k-1}_{(p)}$$
there is a homotopy commutative diagram
$$\xymatrix{\Sigma G_{(p)}\ar[r]^{\rho_2}\ar[d]^{\Sigma\pi_k}&C_\Phi\ar[r]^(.3){\delta_2}\ar[d]^{\lambda_2}&\Sigma^2 S^{2n_i-1}\wedge S^{2n_j-1}\ar@{=}[d]\\
\Sigma S^{2n_k-1}\ar[r]^(.62){\rho_3}&C_{\widehat{\Phi}}\ar[r]^(.3){\delta_3}&\Sigma^2 S^{2n_i-1}\wedge S^{2n_j-1}}$$
whose rows are homotopy cofibrations. Since $\alpha_1$ is detected by the Steenrod operation $\mathcal{P}^1$, the mod $p$ cohomology of $C_{\widehat{\Phi}}$ is given by
$$\widetilde{H}^*(C_{\widehat{\Phi}};\Z/p)=\langle a_{2n_k},a_{2n_i+2n_j}\rangle,\quad\mathcal{P}^1a_{2n_k}=ca_{2n_i+2n_j}$$
such that $\delta_3^*(\Sigma^2u_{2n_i-1}\otimes u_{2n_j-1})=a_{2n_i+2n_j}$ and $\rho_3^*(a_{2n_k})=\Sigma u_{2n_k-1}$.
Then by \eqref{lambda1}, we get $\rho_2^*\circ\lambda_2^*(a_{2n_k})=\rho_2^*\circ\lambda_1^*(\bar{x}_k)$.
By the homotopy cofiber sequence $\Sigma G_{(p)}\xrightarrow{\rho_2}C_\Phi\xrightarrow{\delta_2}\Sigma^2S^{2n_i-1}\wedge S^{2n_j-1}$ one can see that the inclusion $\rho_2\colon\Sigma G_{(p)}\to C_\Phi$ is injective in the mod $p$ cohomology of dimension $2n_k$, and then we obtain $\lambda_2^*(a_{2n_k})=\lambda_1^*(\bar{x}_k)$. 
By \eqref{lambda1}, we also have $\lambda_2^*(a_{2n_i+2n_j})=\lambda_1^*(\bar{x}_i\bar{x}_j)$.
Hence since $\mathcal{P}^1a_{2n_k}=ca_{2n_i+2n_j}$, we get that $\mathcal{P}^1\bar{x}_k$ includes the term $c\bar{x}_i\bar{x}_j$. Thus since $\iota_2^*(x_m)=\bar{x}_m$ for $m=1,\ldots,\ell$, $\mathcal{P}^1x_k$ must include the term $cx_ix_j$.
Therefore we have established the lemma.
\end{proof}

\begin{theorem}
\label{P}
Suppose $p\ge n_\ell-n_1+2$.
Then the Samelson product $\langle\epsilon_i,\epsilon_j\rangle$ in $G_{(p)}$ is nontrivial if and only if for some $k$, $\mathcal{P}^1x_k$ includes the term $cx_ix_j$ with $c\ne 0$.
\end{theorem}

\begin{proof}
By the result of Kumpel \cite{K}, we can choose each $\epsilon_i$ such as $h_i=1$.
Then the composite
$$S^{2n_1-1}\times\cdots\times S^{2n_\ell-1}\xrightarrow{\epsilon_1\times\cdots\times\epsilon_\ell}G_{(p)}\times\cdots\times G_{(p)}\to G_{(p)}$$
induces a $p$-local homotopy equivalence where the second map is the multiplication, and we identify $G_{(p)}$ with $S^{2n_1-1}_{(p)}\times\cdots\times S^{2n_\ell-1}_{(p)}$ by this $p$-local homotopy equivalence.
Under this assumption, $h_i$ is a unit of $\Z_{(p)}$ for any $i$.
By this decomposition, we can find a projection $\pi_k\colon G_{(p)}\to S^{2n_i-1}_{(p)}$ such that $\pi_k^\ast u_{2n_i-1}=e_i$ for each $i$.
By Lemma \ref{criterion2}, if $\mathcal{P}^1x_k$ includes the term $cx_ix_j$ with $c\ne 0$, then the Samelson product $\langle\epsilon_i,\epsilon_j\rangle$ in $G_{(p)}$ is nontrivial.
As in \cite{KK}, if $\langle\epsilon_i,\epsilon_j\rangle$ is nontrivial, then for some $1\le k\le\ell$ we have $n_k+p-1=n_i+n_j$ and $\pi_k\circ\langle\epsilon_i,\epsilon_j\rangle$ is nontrivial.
Again by Lemma \ref{criterion2}, this implies that $\mathcal{P}^1x_k$ includes the term $cx_ix_j$ with $c\ne 0$.
\end{proof}


\section{Proofs of the results}

Let $p$ be an odd prime and $p_i,e_n\in H^*(B\SO(2n)_{(p)};\Z/p)$ be the mod $p$ reduction of the $i$-th universal Pontrjagin class for $i=1,\ldots,n-1$ and the Euler class respectively. Then 
$$H^*(B\SO(2n)_{(p)};\Z/p)=\Z/p[p_1,\ldots,p_{n-1},e_n]$$
and the maps $\epsilon_i$ and $\theta$ correspond to $p_i$ and $e_n$ respectively in the sense of \eqref{epsilon}.
In particular, we take $\epsilon_i$ so that $h_i=1$ for $i\le\tfrac{p-1}{2}$ and $\theta$ so that $(\Sigma\theta)^\ast\circ\iota_1^\ast(e_n)=\Sigma u_{2n-1}$ and $(\Sigma\theta)^\ast\circ\iota_1^\ast(p_i)=0$ for any $i$.

\begin{lemma}
\label{computation_Pp}
The following statements hold.
\begin{enumerate}
\item
The element $\mathcal{P}^1p_i$ does not include the quadratic term $ce_np_j$ $(c\ne0)$ for any $i$ and $j$.
\item
If $p=2n-1$, the element $\mathcal{P}^1p_1$ is decomposable and includes the term $(-1)^{\frac{p-1}{2}}e_n^2$.
\end{enumerate}
\end{lemma}

\begin{proof}
Since $p_i\in H^\ast(B\SO(2n)_{(p)};\mathbb{Z}/p)$ is contained in the image from $H^\ast(B\SO(2n+1)_{(p)};\mathbb{Z}/p)$, if a quadratic term of $\mathcal{P}^1p_i$ includes $e_n$, it must be a multiple of $e_n^2$ and $i=n-\tfrac{p-1}{2}\ge1$. Thus the first statement holds.
Recall that for a maximal torus $T$ of $\SO(2n)$ and the natural map $\iota\colon BT_{(p)}\to B\SO(2n)_{(p)}$, we have
$$H^*(BT_{(p)};\Z/p)=\Z/p[t_1,\ldots,t_n],\quad|t_i|=2$$
such that $\iota^*(p_i)$ is the $i$-th elementary symmetric polynomial in $t_1^2,\ldots,t_n^2$ and $\iota^*(e_n)=t_1\cdots t_n$.
In particular, $\iota$ is injective in the mod $p$ cohomology.
Suppose $p=2n-1$.
We have
$$\iota^\ast(\mathcal{P}^1p_1)=\mathcal{P}^1(t_1^2+\cdots+t_n^2)=2((t_1^2)^{\frac{p+1}{2}}+\cdots+(t_n^2)^{\frac{p+1}{2}}).$$
Then we obtain
$$
\mathcal{P}^1p_1\equiv(-1)^{\tfrac{p-1}{2}}e_n^2\mod(p_1,\ldots,p_{n-1})^2
$$
by the Newton formula. Therefore the second statement holds.
\end{proof}

\begin{lemma}
\label{computation_Pe}
The element $\mathcal{P}^1e_n$ is decomposable and the following congruence hold:
$$\mathcal{P}^1e_n\equiv(-1)^{\frac{p-1}{2}}\tfrac{p-1}{2}e_np_{\frac{p-1}{2}}\mod(p_1,\ldots,p_{n-1})^2.$$
\end{lemma}

\begin{proof}
We set $\iota\colon BT_{(p)}\to B\SO(2n)_{(p)}$ as in the proof of Lemma \ref{computation_Pp}.
We have
$$\iota^*(\mathcal{P}^1e_n)=\mathcal{P}^1\iota^*(e_n)=\mathcal{P}^1(t_1\cdots t_n)=t_1\cdots t_n((t_1^2)^{\frac{p-1}{2}}+\cdots+(t_n^2)^{\frac{p-1}{2}}).$$
Then the proof is completed by the Newton formula.
\end{proof}

\begin{proof}
[Proof of Theorem \ref{main}]
Assume $p>2n-2$.
Since the inclusion $\SO(2n-1)_{(p)}\to\SO(2n)_{(p)}$ has a left homotopy inverse, it follows from Theorem \ref{classical} that the Samelson product $\langle\epsilon_i,\epsilon_j\rangle$ is nontrivial if and only if $2i+2j>p$.
To detect the Samelson products $\langle\epsilon_i,\theta\rangle=\langle\theta,\epsilon_i\rangle$ and $\langle\theta,\theta\rangle$ by Theorem \ref{P}, we need the information about the quadratic terms of $\mathcal{P}^1p_i$ and $\mathcal{P}^1e_n$ including $e_n$.
Now these informations have already been obtained in Lemma \ref{computation_Pp} and \ref{computation_Pe}.
Therefore the proof of Theorem \ref{main} is completed.
\end{proof}

Lemma \ref{computation_Pe} implies nontriviality of the Samelson product $\langle\epsilon_{\tfrac{p-1}{2}},\theta\rangle$ not only when $\SO(2n)$ is $p$-regular but also when $\SO(2n)$ is not $p$-regular as follows.

\begin{corollary}
\label{e_theta}
The Samelson product $\langle\epsilon_{\tfrac{p-1}{2}},\theta\rangle=\langle\theta,\epsilon_{\tfrac{p-1}{2}}\rangle$ in $\pi_{2n+2p-4}(\SO(2n)_{(p)})$ is nontrivial for any odd prime $p$.
More precisely, the image of $\langle\epsilon_{\tfrac{p-1}{2}},\theta\rangle$ under the homomorphism induced by the projection $\SO(2n)_{(p)}\to S^{2n-1}_{(p)}$ generates $\pi_{2n+2p-4}(S^{2n-1}_{(p)})\cong\mathbb{Z}/p$.
\end{corollary}

\begin{proof}
Note that, for the projection $\pi\colon\SO(2n)_{(p)}\to S^{2n-1}_{(p)}$, we have $(\Sigma\pi)^\ast \Sigma u_{2n-1}=\iota_1^\ast(e_n)$.
Then the corollary follows from Lemma \ref{criterion2} and \ref{computation_Pe}.
\end{proof}

We next prove Theorem \ref{normal}. Let $X$ be a homotopy associative H-space with inverse. For maps $\alpha\colon A\to X$ and $\beta\colon B\to X$, let $\{\alpha,\beta\}$ denote the composite
$$A\times B\xrightarrow{\alpha\times\beta}X\times X\to X$$
where the last arrow is the commutator map. Then for the projection $q\colon A\times B\to A\wedge B$, we have $q^*(\langle\alpha,\beta\rangle)=\{\alpha,\beta\}$ and the induced map $q^*\colon[A\wedge B,X]\to[A\times B,X]$ is injective. In particular, $\langle\alpha,\beta\rangle$ is trivial if and only if $\{\alpha,\beta\}$.

\begin{lemma}
[cf. {\cite[Proposition 1]{KK}}]
\label{decomp}
For maps $\varphi_i\colon A_i\to X$ ($i=1,2$) and $\beta\colon B\to X$, if $\{\varphi_2,\beta\}$ is trivial, then 
$$\{\varphi_1\cdot\varphi_2,\beta\}=\{\varphi_1,\beta\}\circ\rho_2$$
where $\rho_i\colon A_1\times A_2\times B\to A_i$ denotes the projection for $i=1,2$.
\end{lemma}

\begin{proof}
In the group of the homotopy set $[A_1\times A_2\times B,X]$, we have
$$\{\varphi_1\cdot\varphi_2,\beta\}=[(\varphi_1\circ\pi_1)\cdot(\varphi_2\circ\pi_2),\beta\circ\pi]$$
where $\pi_i\colon A_1\times A_2\times B\to A_i$ for $i=1,2$ and $\pi\colon A_1\times A_2\times B\to B$ denote the projections and $[-,-]$ means the commutator. In a group $G$, we have
$$[xy,z]=x[y,z]x^{-1}[x,z]$$
for $x,y,z\in G$. Then the proof is completed by $[\varphi_i\circ\pi_i,\beta\circ\pi]=\{\varphi_i,\beta\}\circ\rho_i$.
\end{proof}

\begin{proof}
[Proof of Theorem \ref{normal}]
Let $\iota\colon\SO(2n-1)\to\SO(2n)$ denote the inclusion and $\pi\colon\SO(2n)\to S^{2n-1}$ the projection.
For $p=2$ and $n\ge 2$, as remarked in Section \ref{section_intro}, the $2$-localization $\iota_{(2)}\colon\SO(2n-1)_{(2)}\to\SO(2n)_{(2)}$ is not homotopy normal by the argument by James \cite[Proof of Theorem (3.1)]{J}.

If $2<p\le 2n-1$, then the Samelson product
$$\pi_{(p)}\circ\langle\iota_{(p)},1_{\SO(2n)_{(p)}}\rangle\circ(\epsilon_{\tfrac{p-1}{2}}\wedge\theta)=\pi_{(p)}\circ\langle\epsilon_{\tfrac{p-1}{2}},\theta\rangle$$
is nontrivial in $\pi_{2n+2p-4}(S^{2n-1}_{(p)})$ by Corollary \ref{e_theta}.
This implies that $\iota_{(p)}$ is not homotopy normal.

Suppose $p>2n-1$.
Note that the identity map of $\SO(2n)_{(p)}$ is identified with the map $\iota_{(p)}\cdot\theta\colon\SO(2n-1)_{(p)}\times S^{2n-1}_{(p)}\to\SO(2n)_{(p)}$.
Then it follows from Lemma \ref{decomp} that $\iota_{(p)}$ is homotopy normal if the Samelson product $\langle\iota_{(p)},\theta\rangle$ is trivial.
Note also that $\iota_{(p)}$ is identified with the map $\epsilon_1\cdots\epsilon_{n-1}\colon S^3_{(p)}\times\cdots\times S^{4n-5}_{(p)}\to\SO(2n-1)_{(p)}$. Then it is sufficient to show that $\{\epsilon_1\cdots\epsilon_{n-1},\theta\}$ is trivial. By Lemma \ref{decomp}, this is equivalent to that $\langle\epsilon_i,\theta\rangle$ are trivial for all $i$.
Thus $\iota_{(p)}$ is homotopy normal by Theorem \ref{main}.
\end{proof}

We finally prove Theorem \ref{H-type}. Let $X,Y$ be homotopy associative H-spaces with inverse. Recall that the H-deviation $d(f)$ of a map $f\colon X\to Y$ is defined by
$$d(f)\colon X\wedge X\to Y,\quad(x_1,x_2)\mapsto f(x_1x_2)f(x_2)^{-1}f(x_1)^{-1}.$$
By definition, $f$ is an H-map if and only if the H-deviation $d(f)$ is trivial. 

\begin{lemma}
\label{deviation}
Let $X_1,X_2,Y$ be homotopy associative H-spaces with inverse, and $\lambda_i\colon X_i\to Y$ be H-maps for $i=1,2$. Then the map $\lambda_1\cdot\lambda_2\colon X_1\times X_2\to Y$ is an H-map if and only if the Samelson product $\langle\lambda_1,\lambda_2\rangle$ is trivial.
\end{lemma}

\begin{proof}
For $x_i,x_i'\in X_i$ ($i=1,2$), we have
\begin{align*}
d(\lambda_1\cdot\lambda_2)(x_1,x_2,x_1',x_2')&\simeq\lambda_1(x_1x_1')\lambda_2(x_2x_2')\lambda_2(x_2')^{-1}\lambda_1(x_1')^{-1}\lambda_2(x_2)^{-1}\lambda_1(x_1)^{-1}\\
&\simeq\lambda_1(x_1)(\langle\lambda_1,\lambda_2\rangle(x_1',x_2))\lambda_1(x_1)^{-1}
\end{align*}
since $\lambda_1,\lambda_2$ are H-maps. Then since $\lambda_1$ is an H-map, $d(\lambda_1\cdot\lambda_2)$ is trivial if and only if so is $\langle\lambda_1,\lambda_2\rangle$, completing the proof.
\end{proof}

\begin{proof}
[Proof of Theorem \ref{H-type}]
Obviously, the map $\iota_{(p)}\cdot\theta$ is a homotopy equivalence, so it remains to show that it is an H-map. By definition, we have $d(\theta)\in\pi_{4n-2}(\SO(2n)_{(p)})$, and then by \cite[Proposition 13.6]{T} and $p>2n-1$, $d(\theta)$ is trivial, implying that $\theta$ is an H-map. The inclusion $\iota_{(p)}$ is clearly an H-map, and in the proof of Theorem \ref{normal} the Samelson product $\langle\iota_{(p)},\theta\rangle$ is shown to be trivial for $p>2n-1$. Thus by Lemma \ref{deviation}, $\iota_{(p)}\cdot\theta$ is an H-map. Note that we have not fixed an H-structure of $S^{2n-1}_{(p)}$. There is a one to one correspondence between H-structures on $S^{2n-1}_{(p)}$ and $\pi_{4n-2}(S^{2n-1}_{(p)})$. By \cite[Proposition 13.6]{T} and $p>2n-1$, $\pi_{4n-2}(S^{2n-1}_{(p)})=0$, so there is only one H-structure on $S^{2n-1}_{(p)}$. By \cite{A}, $S^{2n-1}_{(p)}$ has a homotopy associative and homotopy commutative H-structure. Then $S^{2n-1}_{(p)}$ must be  a homotopy associative and homotopy commutative H-space. 
\end{proof}

\end{document}